\documentclass[reqno,12pt]{amsart}
\usepackage{amsmath,amsthm,amsfonts,amssymb,latexsym,enumerate}
\usepackage{color, xcolor}
\usepackage{enumerate, enumitem, verbatim, graphicx}
\usepackage[utf8]{inputenc}

\usepackage{latexsym,enumerate}
\usepackage{amssymb, xcolor,mathrsfs}
\usepackage{amsmath,amsthm,amsfonts,amssymb,latexsym, mathabx,euscript,mathtools}

\usepackage[colorlinks,pagebackref]{hyperref}
\usepackage[english]{babel}

\newtheorem{thm}{Theorem}[section]
\newtheorem{lem}[thm]{Lemma}

\newtheorem{prop}[thm]{Proposition}

\newtheorem{hypothesis}[thm]{Hypothesis}

\newtheorem{mainthm}{Theorem}
\newtheorem{mainconjecture}[mainthm]{Conjecture}

\theoremstyle{definition}

\newtheorem{remark}[thm]{Remark}

\def\nor{\trianglelefteq\,}

\newcommand{\F}{\mathbb{F}}

\newcommand{\Irr}{{{\operatorname{Irr}}}}

\newcommand{\RNum}[1]{\uppercase\expandafter{\romannumeral #1\relax}}
\newcommand{\Al}{\textup{\textsf{A}}}
\newcommand{\Sy}{\textup{\textsf{S}}}

\newcommand{\cod}{\normalfont{\mbox{cod}}}
\newcommand{\cd}{\normalfont{\mbox{cd}}}
\newcommand{\GL}{\normalfont{\mbox{GL}}}
\newcommand{\GU}{\normalfont{\mbox{GU}}}
\newcommand{\SL}{\normalfont{\mbox{SL}}}

\newcommand{\PSL}{\normalfont{\mbox{PSL}}}
\newcommand{\PSp}{\normalfont{\mbox{PSp}}}
\newcommand{\Sp}{\normalfont{\mbox{Sp}}}
\newcommand{\PSU}{\normalfont{\mbox{PSU}}}
\newcommand{\Aut}{\mathrm{Aut}}
\newcommand{\irr}{\mathrm{Irr}}

\newcommand{\bZ}{{\mathbf Z}}

\newcommand{\bC}{{\mathbf{C}}}
\newcommand{\bG}{{\mathbf{G}}}

\newcommand{\bT}{{\mathbf{T}}}

\newcommand{\bO}{{\mathbf{O}}}

\newcommand{\EC}{\mathcal{E}}

\newcommand{\ta}{\hspace{0.5mm}^{2}\hspace*{-0.2mm}}
\newcommand{\tb}{\hspace{0.5mm}^{3}\hspace*{-0.2mm}}

\title[The codegree isomorphism problem for finite simple groups]
{The codegree isomorphism problem\\ for finite simple groups II}

\author[N.\,N. Hung]{Nguyen N. Hung}
\address{Department of Mathematics, The University of Akron, Akron,
OH 44325, USA} \email{hungnguyen@uakron.edu}

\author[A. Moret\'{o}]{Alexander Moret\'{o}}
\address{Departamento de Matem\'aticas, Universidad de Valencia,
46100 Burjassot, Valencia, Spain} \email{alexander.moreto@uv.es}

\thanks{We are indebted to Pham Huu Tiep for several ideas in the proof
of Theorem~\ref{thm:E} for classical groups. We also thank him and
Bob Guralnick for helpful comments on group cohomology and Eamonn
O'Brien for checking some cases of Hypothesis~\ref{hypothesis}. Part
of this work was done while the first author was visiting the
Vietnam Institute for Advanced Study in Mathematics (VIASM), and he
thanks the VIASM for the financial support. The second author is
supported by Ministerio de Ciencia e Innovaci\'on (Grant
PID2022-137612NB-I00 funded by
MCIN/AEI/10.13039/501100011033 and ``ERDF A way of making Europe").
He also acknowledges support by Generalitat Valenciana
CIAICO/2021/163. }


\begin{document}

\maketitle

\begin{abstract}
Let $H$ be a nonabelian finite simple group. Huppert's conjecture
asserts that if $G$ is a finite group with the same set of complex
character degrees as $H$, then $G\cong H\times A$ for some abelian
group $A$. Over the past two decades, several specific cases of this
conjecture have been addressed. Recently, attention has shifted to
the analogous conjecture for character codegrees: if $G$ has the
same set of character codegrees as $H$, then $G\cong H$.
Unfortunately, both problems have primarily been examined on a
case-by-case basis. In this paper and the companion \cite{hmt}, we
present a more unified approach to the codegree conjecture and
confirm it for several families of simple groups.
\end{abstract}

\section{Introduction}

This paper is a sequel to our previous work \cite{hmt}, in which we
investigated the codegree isomorphism problem for finite simple
groups.

Let $G$ be a finite group. The codegree of a character $\chi$ of $G$
is defined as \[\cod(\chi):= |G:\ker(\chi)|/\chi(1).\] The set of
all the codegrees of irreducible characters of $G$ is denoted by
$\cod(G)$, and the corresponding multiset of the codegrees with
multiplicity is denoted by $C(G)$. This $C(G)$ is called the
\emph{group pseudo-algebra} of $G$. We proved in \cite{hmt} that a
finite group $G$ and a simple group $H$ are isomorphic if and only
if they have the same group pseudo-algebra. In this paper, we drop
the multiplicity and turn our attention to the so-called
\emph{codegree isomorphism conjecture} (see Problem~20.79 of the
Kourovka Notebook \cite{Khukhro}), which says that if $H$ is a
nonabelian simple group and $G$ is a finite group with
$\cod(G)=\cod(H)$ then $G\cong H$. In fact, we propose a stronger
version.

\begin{mainconjecture}\label{conj:hup} Let $H$ be a finite nonabelian simple group and $G$ a
finite group. Then $\cod(G)\subseteq\cod(H)$ if and only if $G\cong
H$.\end{mainconjecture}

The original conjecture has been confirmed for several sporadic
simple groups, Suzuki groups, and linear groups in dimensions 2 and
3, see \cite{BahriAkh, Ahanjideh, gkl, LY22}. However, the approach
in these papers has been primarily case-by-case, employing ad-hoc
arguments for each specific group (in the case of sporadic groups)
or family of groups (in the case of Suzuki and linear groups).

We pursue here (and in the companion paper \cite{hmt}) a more
uniform approach. In particular, we prove in Theorem
\ref{thm:last-section} that, except possibly for odd-characteristic
orthogonal groups in even dimension, Conjecture~\ref{conj:hup}
follows from the following.

\begin{mainconjecture}\label{conj:C}
Let $H$ be a nonabelian simple group. Then, for every perfect group
$G$ with an elementary abelian normal $p$-subgroup $N$ for some
prime $p$ dividing $|H|$ such that $G/N=H$ and $N$ is a faithful
irreducible $H$-module, there exists a faithful irreducible
character $\chi$ of $G$ such that $\chi(1)_p<|N|$.
\end{mainconjecture}

This path allows us to confirm Conjecture~\ref{conj:hup} in several
cases.

\begin{mainthm}\label{thm:D}
Both Conjectures~\ref{conj:hup} and \ref{conj:C} hold when $H$ is an
alternating group, a sporadic simple group, or a simple group of the
following Lie types: $\ta B_2$, $\ta G_2$, $\ta F_4$, $G_2$, $\tb
D_4$, $F_4$, $\ta E_6$, $A_{\leq 3}$, $\ta A_{\leq 6}$, $C_2$,
$C_3$, and $B_3$.
\end{mainthm}

When $H$ is a simple group of Lie type, Conjecture~\ref{conj:C} can
be reduced further to the situation where the prime $p$ is the
defining characteristic of $H$, see Proposition~\ref{prop-tech}.

Conjecture~\ref{conj:C} is related to questions on group cohomology
and $p$-groups that we will discuss in more detail in Section
\ref{sec:cohomo}. This approach appears particularly promising and
we hope that it will be studied further by experts in the field.


Our approach requires investigating the character codegrees of
simple (and even quasisimple) groups. The first step is to
demonstrate that if $S$ and $H$ are finite simple groups such that
$\cod(S)\subseteq\cod(H)$, then $S\cong H$. This was achieved in
\cite{hmt}, utilizing some deep results in the representation theory
of finite simple groups, such as the classification of
prime-power-degree representations, bounds for the largest degree of
irreducible representations, and the existence of $p$-defect zero
characters.

Another major step is to show that the codegree set of a nonabelian
simple group is different from that of a nontrivial perfect central
cover of the group. This turns out to be a daunting task. The
following result might be helpful in future attempts to establish
Huppert's original conjecture for character degrees in full
generality. Several ideas in its proof for the case of classical
groups was kindly given to us by Tiep.

\begin{mainthm}\label{thm:E}
Let $S$ be a finite nonabelian simple group and $G$ a quasisimple
group with center $\bZ(G)$ of prime order such that $G/\bZ(G)\cong
S$. Suppose that $S\not\cong P\Omega^\pm_{2n}(q)$ with $n\geq 4$ and
$q$ odd. Then there exists a faithful character $\chi\in\irr(G)$
such that $\chi(1)/|\bZ(G)|$ is not a character degree of $S$.
Consequently, $\cod(G)\nsubseteq \cod(S)$.
\end{mainthm}

We expect Theorem \ref{thm:E} to also hold for even-dimensional
orthogonal groups, but the proof appears to be much more
complicated, see Section~\ref{sec:2} for further comment on this.

The proofs of Theorem~\ref{thm:E} and \ref{thm:D} are respectively
given in Sections~\ref{sec:2} and \ref{sec:4}. Section~\ref{sec:3}
shows that Conjecture~\ref{conj:hup} follows from
Conjecture~\ref{conj:C}. Finally, Section~\ref{sec:cohomo} discusses
some reductions of Conjecture~\ref{conj:C} to a problem on
$p$-groups or a problem on group cohomology.

We conclude this Introduction by noting that we announced the main
results of this paper in Remark 1 of \cite{hmt}. Since then, there
has been a number of preprints dealing with the codegree isomorphism
conjecture for some families of the groups mentioned in
Theorem~\ref{thm:D}, using the results in \cite{hmt}. Nevertheless,
the approach in this paper remains original and independent of
previous papers on the codegree isomorphism conjecture, except for
\cite{hmt}. Given that this paper strongly relies on the complex
Theorem~B of \cite{hmt}, we considered it appropriate to have that
paper reviewed before publishing this one. We believe that the
cohomological approach discussed in Section~\ref{sec:cohomo} is
perhaps the best way to address most of the remaining cases.



\section{Character degrees of quasisimple groups}\label{sec:2}

The aim of this section is to prove Theorem~\ref{thm:E}. That is, if
$S$ be a finite non-abelian simple group and $G$ a quasi-simple
group with center $\bZ(G)$ of prime order such that $G/\bZ(G)\cong
S$, then there exists a faithful character $\chi\in\irr(G)$ such
that
\begin{equation}\label{eq:einstein}
\chi(1)/|\bZ(G)|\notin\cd(S),
\end{equation}
Here, as usual, $\cd(S)$ denotes the set of the degrees of
irreducible characters of $S$.

We begin with the easy case of alternating and sporadic groups.

\begin{prop}\label{prop:quasi-alternating}
Theorem \ref{thm:E} holds when $S$ is an alternating group or a
sporadic simple group.
\end{prop}

\begin{proof}
When $S$ is a sporadic group or an alternating group of degree at
most $9$, the result can be verified using \cite{Conway}.

Assume that $S=\Al_n$ for $n\geq 10$. In particular, $S$ has a
unique nontrivial perfect central cover, which is the Schur double
cover, denoted by $\widehat{\Al}_n$. Assume, to the contrary, that
$\chi(1)/|\bZ(G)|\in\cd(S)$ for every faithful character
$\chi\in\irr(G)$.

Note that the so-called basic representation of $\widehat{\Al}_n$,
the one labeled by the partition $(n)$, has the degree $2^{\lfloor
(n-2)/2\rfloor}$ (see \cite[p.1774]{KT12}, for instance). Therefore,
we obtain that $2^{\lfloor (n-2)/2\rfloor-1}$ would be a character
degree of $\Al_n$. This is impossible, because the only possible
prime-power character degree of $\Al_n$ is $n-1$ (by
\cite[Theorem~5.1]{BBOO01}), but the equation $2^{\lfloor
(n-2)/2\rfloor-1}=n-1$ returns no solutions.
\end{proof}

The proof of Theorem \ref{thm:E} for groups of Lie type relies on
Lusztig's classification of their irreducible characters. We briefly
recall some basics that will be needed in the remaining of this
section, and refer the reader to \cite{Carter85,DM} for the full
account of the theory.

Let $G=\bG^F$ be the group of fixed points under a Frobenius
endomorphism $F$ of a connected reductive algebraic group $\bG$ over
an algebraically closed field of characteristic $p>0$. Let
$(\bG^\ast,F^\ast)$ be the pair in duality with $(\bG,F)$ and let
$G^\ast:={\bG^\ast}^{F^\ast}$. The set $\Irr(G)$ of irreducible
characters of $G$ is partitioned into subsets known as the
(rational) Lusztig series $\EC(G,s)$ associated to various
$G^\ast$-conjugacy classes of semisimple elements $s\in G^\ast$. The
series $\EC(G,s)$ is defined to be the set of irreducible characters
of $G$ occurring in some Deligne-Lusztig character
$R_{\bT}^{\bG}\theta$, where $\bT$ is an $F$-stable maximal torus of
$\bG$ and $\theta\in\Irr(\bT^F)$ such that the geometric conjugacy
class of $(\bT,\theta)$ corresponds to the $G^\ast$-conjugacy class
containing $s$.

According to \cite[Theorem 13.23 and Remark 13.24]{DM}, for each
semisimple element $s\in G^\ast$, there is a bijection
$\chi\leftrightarrow \psi$ between $\EC(G,s)$ and
$\EC(\bC_{G^*}(s),1)$ -- the set of the unipotent characters of the
centralizer $\bC_{G^*}(s)$, such that
\begin{equation}\label{eq:degree-formula}
\chi(1)=|{G^*}:\bC_{G^*}(s)|_{p'}\psi(1).
\end{equation}
Remark that, when $\bC_{G^*}(s)$ is disconnected, unipotent
characters of $\bC_{G^*}(s)$, by definition, are precisely those
characters in $\Irr(\bC_{G^*}(s))$ lying over unipotent characters
of the connected component $(\bC_{\bG^*}(s)^0)^{F^*}$. Once a
bijection is fixed, we will say that $(s^{G^*},\psi)$ is the Jordan
decomposition of $\chi$.

In certain cases, the desired character $\chi$
satisfying\eqref{eq:einstein} can be chosen to be a semisimple
character (see \cite[Definition 2.6.9]{GM20}). Every Lusztig series
$\EC(G,s)$ contains at least one semisimple character. Their degrees
are the same and given by
\[
|{G^*}:\bC_{G^*}(s)|_{p'}.
\]

We record the following well-known lemma that will be useful in
producing faithful irreducible characters of finite groups of Lie
type.

\begin{lem}\label{lem:restriction-center}
Assume the above notation. Let $s\in G^\ast$ be a semisimple
element. Assume that $|\bZ(G)|=|G^\ast/(G^\ast)'|$. Then all
characters in $\EC(G,s)$ restricts trivially to $\bZ(G)$ if and only
if $s\in (G^\ast)'$. More precisely,
\begin{enumerate}
\item if $s\in (G^\ast)'$, then all characters in $\EC(G,s)$ lie over the trivial character of
$\bZ(G)$, and

\item if $s\notin (G^\ast)'$, then all characters in $\EC(G,s)$ lie
over the same nontrivial character of $\bZ(G)$.
\end{enumerate}
\end{lem}

\begin{proof}
See \cite[Lemma 4.4(ii)]{NT13} for the `if' implication, and
\cite[Lemma~5.8(ii)]{Hung24} for the other implication. The fact
that all the characters in one Lusztig series lie over the same
character of $\bZ(G)$ is from \cite[Lemma~2.2]{Malle07}.
\end{proof}

\begin{prop}\label{prop:quasi2}
Theorem \ref{thm:E} holds when $S$ is an exceptional group of Lie
type.
\end{prop}

\begin{proof}
We can ignore the case where $S$ has trivial Schur multiplier. When
$S$ is one of the groups $\{G_2(3), G_2(4), {}^2E_6(2), Suz(8)\}$,
the result can be checked using \cite{Conway}.

It remains to consider \[S=E_6(q),{}^2E_6(q), \text{ or } E_7(q),\]
where \[3 \mid (q-1), 3\mid (q+1), \text{ or } 2\mid (q-1),\]
respectively. The only possible nontrivial cover of $S$ is then the
universal cover, which turns out to be the corresponding finite
reductive group of simply connected type. (Note that ${}^2E_6(2)$ is
already treated above and excluded here as its Schur multiplier is
exceptional and isomorphic to product of cyclic groups of order 2,2,
and 3.) We denote these quasisimple groups by the usual notation
\[S_{sc}=E_6(q)_{sc}, {}^2E_6(q)_{sc}, \text{ or } E_7(q)_{sc}.\]

Note that the dual group of $S_{sc}$ is the corresponding group of
adjoint type, denoted by $S_{ad}$. By
Lemma~\ref{lem:restriction-center}, if $s$ is chosen to be outside
the commutator subgroup $S_{ad}'$ of $S_{ad}$, every character
$\chi$ in the series $\mathcal{E}(S_{sc},(s))$ lie over a nontrivial
character of $\bZ(S_{sc})$, and hence is faithful. Certainly, in our
situation, \[|S_{ad}: S_{ad}'|=|\bZ(S_{sc})|\in\{2,3\},\] and we may
choose the desired semisimple element $s$ of order $|\bZ(S_{sc})|$.
In such case, we can appeal to L\"{u}beck's webpage \cite{Lubeck}
for the information on conjugacy classes of elements of order 2 or 3
in exceptional groups of Lie type and the structure of their
centralizers. Once $s$ is chosen, we compute the degree of a
semisimple character $\chi_{(s)}$ in the series
$\mathcal{E}(S_{sc},s)$ (note that there may be more than one such a
character) by the formula
\[\chi_{(s)}(1)=|S_{ad}:\bC_{S_{ad}}(s)|_{p'},\]
where $p$ is the defining characteristic of $S$. Using
\cite{Lubeck2}, one can check that both $\chi_{(s)}(1)$ and
$\chi_{(s)}(1)/|\bZ(S_{sc})|$ are not members of $\cd(S_{ad})$,
which shows that the latter is not a member of $\cd(S)$, as wanted
as in \eqref{eq:einstein}.

We provide complete arguments in the case $S=E_7(q)$, where $q=p^n$
is some odd prime power, and skip the details in the remaining
cases. We will use cyclotomic polynomials $\Phi_i:=\Phi_i(q)$ to
express the orders of these centralizers. First assume that $q\equiv
1(\bmod\,4)$. Then $E_7(q)_{ad}$ has one involution class $(s)$
outside $E_7(q)$ with centralizer ${}^2A_7(q).2$ of order
\[|\bC_{E_7(q)_{ad}}(s)|=2 q^{28} \Phi_1^4 \Phi_2^7 \Phi_3 \Phi_4^2
\Phi_6^2 \Phi_8 \Phi_{10} \Phi_{14}.\] A semisimple character
$\chi_{(s)}$ associated to $(s)$ of $G=E_7(q)_{sc}$ then has the
degree
\[
\chi_{(s)}(1)=[E_7(q)_{ad}:\bC_{E_7(q)_{ad}}(s)]_{p'}=
\frac{1}{2}\Phi_1^3\Phi_3^2\Phi_5\Phi_6\Phi_7\Phi_9\Phi_{12}\Phi_{18}.
\]
It can be checked that neither $\chi_{(s)}(1)$ nor $\chi_{(s)}(1)/2$
is a character degree of $E_7(q)_{ad}$, so that
$\chi_{(s)}(1)/2\notin\cd(E_7(q))$. Our desired character $\chi$ as
in (\ref{eq:einstein}) now can be chosen to be $\chi_{(s)}$. For the
other case $q\equiv -1(\bmod\,4)$, $E_7(q)_{ad}$ has one involution
class outside $E_7(q)$ with centralizer $A_7(q).2$ and similar
arguments apply.

When $S=E_6(q)$ with $3\mid (q-1)$, the character $\chi$ can be
taken to be a semisimple character associated to a class $(s)\in
E_6(q)_{ad} - E_6(q)$ of elements of order $3$ with centralizer
$A_2(q^3).3$. Finally, when $H={}^2E_6(q)$ with $q>2$ and $3\mid
(q+1)$, $\chi$ can be chosen to be a semisimple character associated
to a class of elements of order $3$ with centralizer
${}^2A_2(q^3).3$.
\end{proof}

In what follows, $\epsilon\in\{\pm\}$ or $\{\pm 1\}$, depending on
the context. Also, $\PSL_n^{+}(q)$ stands for the linear groups and
$\PSU_n^{-}(q)$ stands for the unitary groups.

\begin{prop}\label{prop:quasi3}
Theorem \ref{thm:E} holds when $S=\PSL_n^\epsilon(q)$ with $n\geq 2$
or $\PSp_{2n}(q)$ with $n\geq 2$.
\end{prop}

\begin{proof}
We recall that $G$ is a quasisimple group such that $|\bZ(G)|$ is of
prime order, say $r$, and $G/\bZ(G)\cong S$. We wish to find a
faithful $\chi\in\Irr(G)$ such that $\chi(1)/r \notin \cd(S).$

Consider $S=\PSp_{2n}(q)$ with $n\geq 2$. We may assume that $q$ is
odd and $G=\Sp_{2n}(q)$, so that $r=2$. It is well known that $G$
has a faithful Weil character $\chi$ of degree either $(q^n+ 1)/2$
or $(q^n- 1)/2$. However, $\chi(1)/2=(q^n\pm 1)/4$ is not a
character degree of $S$, as the minimal nontrivial degree of a
projective irreducible character of $S$ is $(q^n-1)/2$ (see
\cite[Theorem~5.2]{Tiep-Zalesskii96}).

Consider $S=\PSL^\epsilon_n(q)$. As the alternating groups have been
treated in Proposition~\ref{prop:quasi-alternating} and the cases
$S=\PSL_3(4)$, $\PSU_4(2)$, $\PSU_4(3)$, or $\PSU_6(2)$ can be
verified directly, we may assume that $G$ is a quotient of
$\SL^\epsilon_n(q)$. In fact, if $|\bZ(G)|=r$, then the cyclic group
$\bZ(\SL^\epsilon_n(q))$, of order $\gcd(n,q-\epsilon)$, has a
unique subgroup $C$ of order $\gcd(n,q-\epsilon)/r$, and $G$ must be
isomorphic to $\SL^\epsilon_n(q)/C$. Now $\SL^\epsilon_n(q)$ has an
irreducible Weil character $\chi$ of degree
$\chi(1)=(q^n-\epsilon^n)/(q-\epsilon)$ whose kernel is precisely
$C$, see \cite{Tiep15} and \cite{Sze} for the character formulas of
the Weil characters of the linear and unitary groups, respectively.
So $\chi$ becomes a faithful character of $G$. An easy inspection
from \cite[Table II]{Tiep-Zalesskii96} shows that $\chi(1)/2$ is
smaller than the minimal nontrivial character degree of $S$, and so
$\chi(1)/r\notin \cd(S)$, as desired.
\end{proof}

We set up some notation for the proof of Theorem~\ref{thm:E} in the
case where $S$ is an odd-dimensional orthogonal group.

Let $\langle.|.\rangle$ be a non-degenerate symplectic form on the
space $V:=\mathbb{F}_q^{2n}$, where $q$ is odd. The \emph{conformal
symplectic group} $\mathrm{CSp}_{2n}(q)$ is defined to be
\[
\mathrm{CSp}_{2n}(q):=\{g \in \GL(V): \exists \,\tau(g)\in
\mathbb{F}^\ast_q, \langle gu|gv\rangle = \tau(g)\langle u|v\rangle
\,\forall\, u,v\in V\}.
\]
We have $\bZ(\mathrm{CSp}_{2n}(q))\cong C_{q-1}$ and the
\emph{projective conformal symplectic group} is
\[
\mathrm{PCSp}_{2n}(q):=
\mathrm{CSp}_{2n}(q)/\bZ(\mathrm{CSp}_{2n}(q)).
\]
This is indeed the type-$C_n$ finite reductive group of adjoint
type, see \cite[p. 40]{Carter85}. Set \[H:=\mathrm{CSp}_{2n}(q)
\text{ and } G^*:=\mathrm{PCSp}_{2n}(q).\] The group $G^*$ is dual
to \[G:=Spin_{2n+1}(q),\] the type-$B_n$ group of simply connected
type, see \cite[p.120]{Carter85}. Let $\bG^\ast$ and $F^*$ be
respectively a simple algebraic group and associated Frobenius map
defining $G^\ast$, so that ${\bG^\ast}^{F^*}=G^\ast$.

Now, let $s$ be a semisimple element in $G^*$. Consider a pre-image
of $s$, which we will denote by the same $s$, in $H$. From the
aforementioned Jordan decomposition, to control characters of $G$,
we wish to understand the centralizer \[C^*:=\bC_{G^*}(s).\]

Note that $\bC_{\bG^*}(s)$ is not always connected. Also, the
$F^*$-fixed points of the connected component $\bC_{\bG^*}(s)^0$,
which we denote by ${C^*}^0$, is precisely the image of $\bC_{H}(s)$
under the natural projection $\pi:H\rightarrow G^*$.

\begin{lem}\label{lem:index}
Assume the above notation. Then the index of ${C^*}^0$ in $C^*$ is
either 1 or 2.
\end{lem}

\begin{proof}
Let $g\in H$ be such that its image $\pi(g)$ under $\pi$ belongs to
$C^*$; that is, $gsg^{-1}=\lambda s$ for some $\lambda \in
\mathbb{F}^*_q$. Then, for every $u,v\in V$,
\begin{align*}
\lambda^2\tau(s)\langle u|v\rangle=&\langle \lambda su|\lambda
sv\rangle=\langle
gsg^{-1}u|gsg^{-1}v\rangle\\
=&\tau(g)\tau(s)\tau(g^{-1})\langle u|v\rangle=\tau(s)\langle
u,v\rangle,
\end{align*}
and so $\lambda\in\{\pm 1\}$. Thus $\bC_{H}(s)$ has index 1 or 2 in
$\pi^{-1}(C^*)$, proving the lemma.
\end{proof}

\begin{remark}
Note that, for certain semisimple element $s\in G^*$, the case
$\lambda=-1$ in the proof of Lemma~\ref{lem:index} does not occur.
In such situation, $\bC_{\bG^*}(s)$ is connected and $\bC_H(s)$ is
the full pre-image of $C^*={C^*}^0$.
\end{remark}

\begin{lem}\label{lem:odd-dim-orthogonal}
Let $q$ be a power of an odd prime $p$. Let $G^*$ be the projective
conformal symplectic group $\mathrm{PCSp}_{2n}(q)$ with $n\geq 3$
and $s$ a semisimple element of $G^*$. Then the only possible
degrees prime to $p$ of a unipotent character of $\bC_{G^*}(s)$ are
1 and 2.
\end{lem}

\begin{proof}
As above, we denote a pre-image of $s$ under the projection $\pi$ by
the same notation $s$. By \cite[Lemma~2.4]{Nguyen10}, we have
$$\bC_{\mathrm{Sp}_{2n}(q)}(s)\cong \Sp_{2k}(q)\times \Sp_{2(m-k)}(q)
\times\prod_{i=1}^t\GL_{a_i}^{\alpha_i}(q^{k_i}),$$ where $0\leq
k\leq m\leq n$, $\alpha_i=\pm$, and $\sum_{i=1}^tk_ia_i=n-m$, if
$\tau(s)$ is a square in $\mathbb{F}_q$, and
$$\bC_{\mathrm{Sp}_{2n}(q)}(s)\cong \Sp_m(q^2)\times\prod_{i=1}^t\GL_{a_i}^{\alpha_i}(q^{k_i}),$$
where $m$ is even, $\alpha_i=\pm$, and $\sum_{i=1}^tk_ia_i=n-m$, if
$\tau(s)$ is not a square in $\mathbb{F}_q$. By
\cite[Theorem~6.8]{Malle07} and the assumption on $q$, the only
unipotent character of either $\Sp_{-}(q)$ or $\SL^\pm_{-}(q)$ of
degree prime to $p$ is the trivial character. It follows that, the
only unipotent characters of $\bC_{\mathrm{Sp}_{2n}(q)}(s)$ of
degree prime to $p$ are linear characters.

Let ${C^*}^0:={\bC_{\bG^*}(s)^0}^{F^*}$, as above. Using
\cite[Proposition~13.20]{DM} for the embedding
$f:\bC_{\mathrm{Sp}_{2n}(q)}(s)\hookrightarrow
\bC_{\mathrm{CSp}_{2n}(q)}(s)$ and the projection $\pi:
\bC_{\mathrm{CSp}_{2n}(q)}(s)\rightarrow {C^*}^0$, we know that the
degrees of unipotent characters of $\bC_{\mathrm{Sp}_{2n}(q)}(s)$
and ${C^*}^0$ are the same. Therefore, the only unipotent characters
of ${C^*}^0$ of degree prime to $p$ are linear characters.

The result now follows from Lemma~\ref{lem:index} that ${C^*}^0$ has
index 1 or 2 in $\bC_{G^*}(s)$.
\end{proof}

\begin{prop}\label{prop:quasi4}
Theorem \ref{thm:E} holds when $S=\Omega_{2n+1}(q)$ with $n\geq 3$.
\end{prop}

\begin{proof}
The theorem can be easily checked for $\Omega_7(2)$ and
$\Omega_7(3)$, those groups with an exceptional Schur multiplier.
Also, $\Omega_{2n+1}(2^m)\cong \Sp_{2n}(2^m)$ is already considered
in Proposition~\ref{prop:quasi3}. We therefore may assume that
$G=Spin_{2n+1}(q)$ with $q$ odd, which is the (double) universal
cover of $S$. Let $G^\ast:=\mathrm{PCSp}_{2n}(q)$ and
$H:=\mathrm{CSp}_{2n}(q)$, as above. Also, let $\bG^\ast$ and $F^*$
be respectively a simple algebraic group and associated Frobenius
map defining $G^\ast$, so ${\bG^\ast}^{F^*}=G^\ast$.

Suppose that $q\equiv \epsilon\in\{\pm 1\} (\bmod\,4)$. According to
\cite[Table 4.5.1]{gls}, $G^\ast$ contains an
 involution, say $t_\epsilon$, that is not in
 $[G^\ast,G^\ast]=\mathrm{PSp}_{2n}(q)$. Furthermore, if
 $C^\ast:=\bC_{G^\ast}(t_\epsilon)$ and $L^\ast:=\bO^{p'}(C^\ast)$
 where $p$ is the defining characteristic of $S$, then
 \[ L^\ast:=\bO^{p'}(C^\ast)=\SL_n^\epsilon(q)/C_{(n,2)} \text{ and } \bC_{C^\ast}(L^\ast)=C_{q-\epsilon}.\]

Let $\bC:=\bC_{\bG^\ast}(t_\epsilon)$ and $\bC^0$ be its connected
components containing the identity element. We have $\bC/\bC^0\cong
C_2$ and $\bC^0$ is of type $A_{n-1}T_1$, according to
\cite[Table~4.3.1]{gls}. The $F$-fixed-point group
${C^\ast}^0:={\bC^0}^F$ is therefore isomorphic to $A_{n-1}.C_{q-1}$
when $\epsilon=1$ and $^2A_{n-1}.C_{q+1}$ when $\epsilon=-1$.
Moreover, $C^\ast/{C^\ast}^0\cong \bC/\bC^0\cong C_2$ (see
\cite[Theorem~4.2.2(i)]{gls}). All together, we have
\[|C^\ast|=2q^{\frac{1}{2}n(n-1)}\prod_{i=1}^n(q^i-\epsilon^i).\]
Now, the Lusztig series defined by $(t_\epsilon)$ contains a
(semisimple) character of degree $D/2$, where
\[D:=2[G^\ast:C^\ast]_{p'}=\frac{\prod_{i=1}^n (q^{2i}-1)}{\prod_{i=1}^n (q^i-\epsilon^i)}.\]
Note that, as $t_\epsilon$ does not belong to $[G^\ast,G^\ast]$,
this character is faithful, again by
Lemma~\ref{lem:restriction-center}.

We claim that $D/4$ is not a character degree of $S$, and this will
complete the proof. Suppose otherwise, i.e., there is $\chi\in
\Irr(S)$ with $\chi(1)=D/4$. Let $s\in G^*=\mathrm{PCSp}_{2n}(q)$ be
a semisimple element labeling the Lusztig series containing $\chi$.
In fact, \[s\in [G^*,G^*]=\PSp_{2n}(q),\] by
Lemma~\ref{lem:restriction-center}. Note that $D/4$ is coprime to
$p$. Formula \eqref{eq:degree-formula} implies that the degree of
the unipotent character $\psi$ of $\bC_{G^*}(s)$ in Lusztig's Jordan
decomposition is coprime to $p$ as well. By
Lemma~\ref{lem:odd-dim-orthogonal}, we have $\psi(1)=1$ or $2$. In
summary, we obtain
\[
\chi(1)=\frac{D}{4}=\alpha\cdot[G^\ast:\bC_{G^\ast}(s)]_{p'}
\]
for some $\alpha\in\{1,2\}$. It follows that
\[
|\bC_{G^\ast}(s)|_{p'}=4\alpha|\GL_n^\pm(q)|_{p'}.
\]

Again, for notational convenience, we will identify $s$ with a
pre-image of $s$ in $\Sp_{2n}(q)$. Observe that
\[|\bC_{\mathrm{PCSp}_{2n}(q)}(s)|=\frac{\kappa}{q-1}|\bC_{\mathrm{CSp}_{2n}(q)}(s)|\]
for some divisor $\kappa$ of $q-1$. In fact, Lemma~\ref{lem:index}
shows that $\kappa$ is a divisor of $2$.

Now,
$|\bC_{\mathrm{CSp}_{2n}(q)}(s)|=(q-1)|\bC_{\mathrm{Sp}_{2n}(q)}(s)|$,
by \cite[Lemma~2.4]{Nguyen10}. It follows that,
\[
|\bC_{\mathrm{Sp}_{2n}(q)}(s)|_{p'}=\frac{4\alpha}{\kappa}|\GL_n^\pm(q)|_{p'},
\]
and thus
\begin{equation}|\bC_{\mathrm{Sp}_{2n}(q)}(s)|_{p'}=\{2 \text { or }4 \text { or
}8\}|\GL_n^\pm(q)|_{p'}.\label{eq:1}\end{equation}

Recall from \cite[Lemma 2.4]{Nguyen10} that
$$\bC_{\mathrm{Sp}_{2n}(q)}(s)\cong \Sp_{2k}(q)\times \Sp_{2(m-k)}(q)
\times\prod_{i=1}^t\GL_{a_i}^{\alpha_i}(q^{k_i}),$$ for some $0\leq
k\leq m\leq n$, $\alpha_i=\pm$, and $\sum_{i=1}^tk_ia_i=n-m$.

Suppose first that the group on the right-hand side of (\ref{eq:1})
is unitary. Suppose furthermore that $n$ is odd. Zsigmondy's theorem
then implies that $t=1$, $k_1a_1=n$, and in fact
$\bC_{\mathrm{Sp}_{2n}(q)}(s)\cong \GU_{a_1}(q^{k_1})$, which in
turn implies that $|\bC_{\mathrm{Sp}_{2n}(q)}(s)|_{p'}\leq
|\GU_n(q)|_{p'}$, violating (\ref{eq:1}). Let $n$ be even. We see
that $q^{n-1}+1$ is a factor of
$|\bC_{\mathrm{Sp}_{2n}(q)}(s)|_{p'}$. Using a primitive prime
divisor of $q^{2n-2}-1$, we deduce that either $\Sp_{2n-2}(q)$ or
$\GU_{a_i}(q^{k_i})$ for some relevant $a_i$ and $k_i$ with $k_ia_i=
n-1$ must be a factor of $\bC_{\mathrm{Sp}_{2n}(q)}(s)$. (Note that
the case $\bC_{\mathrm{Sp}_{2n}(q)}(s)\cong \GU_{a_1}(q^{k_1})$ is
already eliminated.) In either case, there is only one other direct
factor of $\bC_{\mathrm{Sp}_{2n}(q)}(s)$, which must be either
$\Sp_2(q)$ or $\GL_1^\pm(q)$. It is now straightforward to check
that none of these cases can fulfill (\ref{eq:1}). The case of a
linear group in (\ref{eq:1}) is entirely similar.
\end{proof}

\begin{prop}\label{prop:quasi5}
Theorem \ref{thm:E} holds when $S=\Omega^\pm_{2n}(q)$ with $n\geq 4$
and $q$ even.
\end{prop}

\begin{proof}
Simple even-dimensional orthogonal groups in even characteristic
have trivial Schur multiplier, unless the group is $\Omega^+_8(2)$.
The Schur multiplier $\Omega^+_8(2)$ is the Klein four-group, and
one can verify the result directly using \cite{Conway}.
\end{proof}

We have completed the proof of Theorem~\ref{thm:E}.

The case of even-dimensional orthogonal groups in odd characteristic
appears to be complicated. On one hand, the Schur multiplier of $S$
is of order 4, and hence our group $G$ (in Theorem~\ref{thm:E}) is
not the entire Schur cover of $S$.
Lemma~\ref{lem:restriction-center} is therefore not enough to
produce a desired character. On the other hand, it is much more
difficult than the odd-dimensional case to control the centralizer
of a semisimple element in the dual group $G^*$, which in this case
is the projective conformal orthogonal group
$\mathrm{P}(\mathrm{CO}^\pm_{2n}(q)^0)$.


\section{Conjecture \ref{conj:C} implies
Conjecture~\ref{conj:hup}}\label{sec:3}

In this short section we show that, if Conjecture~\ref{conj:C} holds
for a simple group $H$ that is not an even-dimensional orthogonal
group in odd characteristic, then Conjecture~\ref{conj:hup} holds
for $H$.

\begin{thm}\label{thm:last-section}
Let $H$ be a non-abelian simple group. Suppose that $H\not\cong
P\Omega^\pm_{2n}(q)$ with $q$ odd and $n\geq 4$. If
Conjecture~\ref{conj:C} holds for $H$, then
Conjecture~\ref{conj:hup} holds for $H$.
\end{thm}

\begin{proof}
Let $G$ be a minimal counterexample to Conjecture~\ref{conj:hup}.
Since $H$ is perfect, \cite[Theorem~2.3]{hmt} implies that $G$ is
perfect. Let $N$ be a maximal normal subgroup of $G$. Since
$\cod(G/N)\subseteq\cod(G)\subseteq\cod(H)$, and $G/N$ is simple, it
follows from \cite[Theorem~B]{hmt} that $G/N\cong H$. Furthermore,
by the minimality of $G$ as a counterexample, we have that $N$ is a
minimal normal subgroup of $G$. Also, $N$ is the unique minimal
normal subgroup of $G$ since, otherwise, $G=H\times H$, violating
the assumption $\cod(G)\not\subseteq\cod(H)$. In particular, every
character in $\Irr(G|N)$ is faithful.

We claim that $N$ is an elementary abelian $p$-group for some prime
$p$ that divides $|H|$. Suppose that $N$ is not abelian. Then there
exists a non-principal character in $\Irr(N)$ that extends to
$\chi\in\Irr(G)$ (see \cite[Theorem~7.1]{mor22} for instance).
Clearly, $\chi(1)<|N|$. Hence $\cod(\chi)>|G/N|=|H|$, which
contradicts $\cod(G)\subseteq\cod(H)$. We have shown that $N$ is an
elementary abelian $p$-group. Let $\pi(X)$ denote the set of prime
divisors of $|X|$. Since
$\pi(G)=\pi(\cod(G))\subseteq\pi(\cod(H))=\pi(H)$, we deduce that
$p$ divides $|H|$, as claimed.

Note that $\bC_G(N)=N$ or $\bC_G(N)=G$. Assume first that
$\bC_G(N)=G$. Then $N=\bZ(G)$ and, as $G$ is perfect and $N$ is the
unique minimal normal subgroup of $G$, $G$ must be a quasisimple
group with the center $\bZ(G)$ of prime order. We are therefore done
using Theorem~\ref{thm:E}.

It remains to consider the case $\bC_G(N)=N$. In particular, $N$ is
a faithful irreducible $H$-module. By Conjecture~\ref{conj:C}, there
exists $\chi\in\Irr(G)$ faithful such that $\chi(1)_p<|N|$.
Therefore, $\cod(\chi)_p>|G|_p/|N|=|H|_p$. It follows that
$\cod(\chi)\not\in\cod(H)$, which contradicts the assumption
$\cod(G)\subseteq\cod(H)$.
\end{proof}

Note that if one could prove Theorem~\ref{thm:E} for
even-dimensional orthogonal groups, then the above argument also
shows that Conjecture~\ref{conj:C} implies Conjecture~\ref{conj:hup}
for these groups.

\section{Theorem \ref{thm:D}}\label{sec:4}

We now work toward a proof of Conjecture~\ref{conj:C}. The following
simple observation will be useful.

\begin{lem}
\label{lem-proj} Let $G$ be a finite group and $N\trianglelefteq G$.
Let $p$ be a prime divisor of $|G/N|$ and let $P/N$ be a Sylow
$p$-subgroup of $G/N$.
\begin{enumerate}
\item Suppose that $\theta\in\Irr(N)$ is $G$-invariant. Then there exists
$\chi\in\Irr(G|\theta)$ such that
\[\left(\chi(1)/\theta(1)\right)_p\leq\sqrt{|G/N|_p}.\]

\item Suppose that $\theta\in\Irr(N)$ is $P$-invariant. Then there exists
$\chi\in\Irr(G|\theta)$ such that
\[\left(\chi(1)/\theta(1)\right)_p\leq\sqrt{|G/N|_p}.\]
\end{enumerate}
\end{lem}

\begin{proof}
We first prove (1). Replacing $(G,N,\theta)$ by an isomorphic
character triple (see Chapter 11 of \cite{Isaacs}), we may assume
that $N$ is central and $\theta$ is linear. Let  $m_p(G)$ be the
smallest $p$-part of the degrees of the characters in
$\Irr(G|\theta)$. Since
$|G/N|=\sum_{\chi\in\Irr(G|\theta)}\chi(1)^2$ is a multiple of
$m_p(G)^2$, we deduce that $m_p(G)\leq\sqrt{|G/N|_p}$, as desired.

For (2), let $T=I_G(\theta)$ - the inertia subgroup of $\theta$ in
$G$, so that $P\leq T$. By part (1), there exists
$\psi\in\Irr(T|\theta)$ such that
\[\left(\psi(1)/\theta(1)\right)_p\leq\sqrt{|T/N|_p}=\sqrt{|G/N|_p}.\]
Now, by Clifford's correspondence, $\chi=\psi^G\in\Irr(G|\theta)$.
Notice that since $|G:T|$ is a $p'$-number,  $\chi(1)_p=\psi(1)_p$,
and so
\[\left(\chi(1)/\theta(1)\right)_p\leq\sqrt{|G/N|_p},\]
as desired.
\end{proof}

The following result will allow us confirm Conjecture~\ref{conj:C}
in many cases.

\begin{prop}
\label{prop-bra} Conjecture~\ref{conj:C} holds if
$|N|>\sqrt{|H|_p}$. In particular, Conjecture~\ref{conj:C} holds if
$p^{2d}>|H|_p$, where $d$ is the smallest degree of the non-trivial
irreducible $p$-Brauer characters of $H$.
\end{prop}

\begin{proof}
Consider the action of $G$ on the $p$-group $\Irr(N)$. It is clear
that there exists $\mathbf{1}_N\neq\lambda\in\Irr(N)$ in an orbit of
$p'$-size. Set $T:=I_G(\lambda)$, so that $T$ contains a Sylow
$p$-subgroup of $G$. By Lemma~\ref{lem-proj}, there exists
$\chi\in\Irr(G|\lambda)$ such that
$$
\chi(1)_p\leq\sqrt{|G/N|_p}=\sqrt{|H|_p}<|N|,
$$
as wanted.

In order to see the second part, note that since $N$ is  a faithful
irreducible $H$-module, $\dim N\geq d$. Hence $|N|\geq
p^d>\sqrt{|H|_p}$.
\end{proof}

The following completes the proof of Theorem~\ref{thm:D}. We will
write $d(H)$ to denote the smallest degree of the nontrivial irreducible Brauer
characters of $H$ (the prime will be understood from the context).

\begin{prop}
\label{prop-tech} Let $H$ be a nonabelian simple group. Let $G$ be a
perfect group with a normal elementary abelian $p$-subgroup $N$ for
some prime $p$ dividing $|H|$ such that $G/N=H$ and $N$ is a
faithful irreducible $H$-module. Suppose that $H$ belongs to one of
the following:
\begin{enumerate}
\item
sporadic groups,
\item
alternating groups,
\item
groups of Lie type in characteristic $\ell\neq p$,
\item
exceptional group of Lie type in characteristic $p$: $\ta B_2(q)$,
$\ta G_2(q)$, $\ta F_4(q)$, $G_2(q)$, $\ta D_4(q)$, $F_4(q)$, $\ta
E_6(q)$,
\item
classical groups in characteristic $p$: $\PSL_{\leq 4}(q)$,
$\PSU_{\leq 7}(q)$, $\PSp_4(q)$, $\PSp_6(q)$,
$\mathrm{P}\Omega_7(q)$, $\mathrm{P}\Omega_8^\pm(q)$.
\end{enumerate}
Then there exists $\chi\in\Irr(G)$ faithful such that
$\chi(1)_p<|N|$.
\end{prop}

\begin{proof}
Let $d(H)$ denote the smallest degree of nontrivial irreducible
$p$-Brauer characters of $H$. By Proposition \ref{prop-bra}, the
result follows for $H$ if we are able to prove that \[
p^{2d(H)}>|H|_p.\]

\medskip

(i) The smallest degrees of the nontrivial $p$-Brauer characters of
the sporadic groups are collected in Table 1 of \cite{jan}. It is
routine to check that the desired inequality is satisfied.

\medskip

(ii) We now consider the alternating groups. It is well-known that
irreducible $\mathbb{F}_p\Sy_n$-modules are in one-to-one bijection
with $p$-regular partitions of $n$. Let $\lambda$ be such a
partition and $D^\lambda$ denote its corresponding module. When the
first part of $\lambda$ is at least $n-4$, an explicit lower bound
for $\dim D^\lambda$ was obtained by James \cite{James83}. On the
other hand, it was shown by M\"{u}ller in \cite[\S 6.2]{Mul16} that,
for $n\geq 11$ and $\lambda$ having the first part at most $n-4$,
\[
\dim D^\lambda \geq
\min\left\{\frac{n^4-14n^3+47n^2-34n}{24},g(n)\right\},
\]
where
\[
g(n):=\left\{\begin{array}{ll}55\cdot 2^{(n-11)/2}& \mathrm{ if }\ n \ \mathrm{ is \ odd}, \\
89\cdot 2^{(n-12)/2} & \mathrm{if}\ n\ \mathrm{is\ even}.\end{array}
\right.
\]
Note that $|\Al_n|_p\leq p^{(n-1)/(p-1)}$ for $p$ odd and
$|\Al_n|_2\leq 2^{n-2}$. These bounds are sufficient to verify the
inequality for alternating groups of degree, of course, $n\geq 11$.
The case of smaller $n$ can be checked using \cite{Atlas2}.

\medskip

(iii) Next, consider the case where $H$ is a simple group of Lie
type in characteristic $\ell\neq p$. A good bound for the smallest
degree $d(H)$ of the non-trivial irreducible $p$-Brauer characters
of $H$ was obtained by Land\'{a}zuri, Seitz, and Zalesskii
\cite{LS74,SZ93} (see also \cite[Table~I]{Tiep-Zalesskii96}). (Note
that this bound does not depend on $p$.) From this it is easy to
check that \[p^{2d(H)}> |H|_\ell\] for all $H$. Therefore, we are
done if $\ell$ is a dominant prime in $|H|$ (which means that
$|H|_\ell \geq |H|_p$ for every $p$). If otherwise, by \cite[Theorem
3.3]{Kimmerle-et-al}, $H$ is either $\PSL_2(\ell)$ with $\ell$ a
Mersenne prime, $\PSL_2(2^a)$ with $2^a+1$ a Fermat prime,
$\PSL_2(8)$, or $\PSU_3(3)$, in which cases the desired inequality
can be easily verified.

\medskip

(iv) Now, assume that $H$ is a simple group of Lie type in
characteristic $p$. Recall that, By Proposition \ref{prop-bra}, we
are done if $|N|>\sqrt{|H|_p}$. Assume by contradiction that
\[|N|\leq\sqrt{|H|_p}.\]
Let $\lambda\in\irr(N)\backslash \{\mathbf{1}_N\}$ and
$T:=I_G(\lambda)$. Since $|G:T|$ is the number of $G$-conjugates of
$\lambda$, we have $|G:T|<|N|$, and it follows that
\[|G:T|<\sqrt{|H|_p}.\]

If $T=G$ then $\ker(\lambda) \nor G$, and thus $\ker(\lambda)=1$,
implying that $N$ is cyclic, which cannot happen as $H$ can be
embedded into $\Aut(N)$. So $T$ is a proper subgroup of $G$, and
thus $T/N$ is a proper subgroup of $H=G/N$ as well. Let $M$ be a
maximal subgroup of $G/N$ such that $T/N \subseteq M$. Then
\[|G:T|\geq |H:M|.\]

We now have $|H:M|<\sqrt{|H|_p}$, where $M$ is a certain maximal
subgroup of $H$ and $p$ is the defining characteristic of $H$. Note
that $H$ naturally permutes the right cosets of $M$, so $[H:M]$ is
the degree of a permutation representation of $H$. Minimal degrees
of permutation representations of simple groups of Lie type have
been worked out by various authors and are available, for example,
in \cite[Table 4]{GMPS15}. The data shows that, for the listed
groups in (iv) and (v), the minimal degree for $H$ always exceeds
$\sqrt{|H|_p}$. This contradiction completes the proof.
\end{proof}

Theorem \ref{thm:D} follows from Theorem \ref{thm:last-section},
Lemma~\ref{lem-proj}, Proposition~\ref{prop-bra}, and
Proposition~\ref{prop-tech}. Since we have proved
Conjecture~\ref{conj:C} for $D_4$ and $\ta D_4$ in
Proposition~\ref{prop-tech}, Conjecture~\ref{conj:hup} for these
groups would follow from a proof of Theorem~\ref{thm:E} for them.


\section{Further reductions}
\label{sec:cohomo}

In this section we discuss some further possible reductions for
Conjectures~\ref{conj:C} and \ref{conj:hup}.

The split-extension case of  Conjecture~\ref{conj:C} is elementary.
Note that the conclusion of the following is much stronger than
needed.

\begin{prop}
\label{prop-split} Let $G=HN$, where $H$ is simple non-abelian and
$N$ is a normal  elementary abelian $p$-subgroup of $G$. Suppose
that the (conjugation) action of $H$ on $N$ is faithful and
irreducible. Then there exists a faithful character $\chi\in\Irr(G)$
of degree prime to $p$.
\end{prop}

\begin{proof}
Since $H$ acts on the $p$-group $\Irr(N)$,  there exists a
non-principal character $\lambda\in\Irr(N)$ in an $H$-orbit of
$p'$-size.  Thus $T:=I_G(\lambda)$ contains a Sylow $p$-subgroup of
$G$. By \cite[Problem~6.18]{Isaacs}, the character $\lambda$ extends
to every Sylow subgroup of $G$. Now, \cite[Corollary~11.31]{Isaacs}
implies that there exists $\tilde{\lambda}\in\Irr(T|\lambda)$ that
extends $\lambda$. Hence, by the Clifford correspondence,
$\tilde{\lambda}^G$ is a faithful irreducible character of $G$ of
$p'$-degree, as wanted.
\end{proof}

Given a simple group $H$ of classical Lie type, we will refer to the
natural module of the full cover $\tilde{H}$ of $H$ as the natural
module of $H$. We will use the fact, following from
\cite[Proposition~5.4.11]{kl}, that if $\varphi$ is an irreducible
Brauer character of $H$ with $\varphi(1)>d$, where $d$ is the
dimension of the natural module for $H$, then
$\varphi(1)>\log_p\sqrt{|H|_p}$.

\begin{prop}
\label{prop-red1} Let $H$ be a classical simple group of Lie type
over the field with $p$ elements, where $p$ is a prime. Suppose
that, for every perfect group $G$ with a normal subgroup $N$ such
that \begin{enumerate}
\item $G/N=H$,
\item $N$ is an elementary abelian $p$-group of order equal to the size of the
natural module for $H$,
\item $N$ is a faithful irreducible
$H$-module, and
\item $G$ is a non-split extension of $N$ by $H$,
\end{enumerate}
there exists a faithful character $\chi\in\Irr(G)$ such that
$\chi(1)_p<|N|$. Then Conjecture~\ref{conj:C} holds for $H$.
\end{prop}

\begin{proof}
By Proposition \ref{prop-split} and the hypothesis, we may assume
that the extension is non-split and that $|N|$ exceeds the size of
the natural module for $H$. Since  $N$ is a faithful irreducible
$\F_pH$-module, Corollary~3F of \cite{won} implies that $N$ is
absolutely irreducible.
By hypothesis, we may assume that $\dim_{\F_p} N>d({H})$, which is
the dimension of the natural module for ${H}$.  But then
\cite[Proposition 5.4.11]{kl}
 implies that $|N|>\sqrt{|H|_p}$
and the result follows from Proposition~\ref{prop-bra}.
\end{proof}

We believe that it should be possible to remove the hypothesis on
the order of the field in Proposition~\ref{prop-red1}. As pointed
out to us by Guralnick and Tiep, the second cohomology group
$H^2(G,V)$, where $G$ is a classical group of simply connected type
and $V$ is the natural module, is most of the time vanishing (see
\cite{avr, dem73, gri}, for instance). When $H^2(G,V)$ does vanish,
Conjecture~\ref{conj:C} would follow from
Proposition~\ref{prop-split}. This applies, for example, when
$H=\PSL_n(2)$ for $n>5$, as shown in \cite{dem73}.
\begin{thm}
Conjecture \ref{conj:hup} holds when $H=\PSL_n(2)$ for $n>5$.
\end{thm}

This provides the first family of groups of Lie type with
arbitrarily large rank for which either of the original Huppert's
conjecture on character degrees or the codegree counterpart is
proved. We expect these ideas to work for other families of simple
groups of Lie type with large rank.


We now present another reduction for Conjecture~\ref{conj:hup} in
the case of classical groups to a problem on $p$-groups. We remark
that, while it is common to reduce problems on finite groups to
simple groups, it is unusual to reduce problems concerning groups
that are close to simple to $p$-groups, which lie at the opposite
end of the spectrum of finite groups.

\begin{hypothesis}\label{hypothesis}
Let $p$ be a prime and let $H$ be a simple classical group over the
field with $q=p^a$ elements. Suppose that $Z$ is a normal (central)
subgroup of order $p$ of a $p$-group $P$ and $P/Z$ is a Sylow
$p$-subgroup of  $H$. Then there exists $\psi\in\Irr(P|Z)$ such that
its degree is less than the size of the natural $\F_q$-module for
$H$.
\end{hypothesis}

In the following, given a normal subgroup $N$ of a group $G$, we write $\Irr(G|N)$ to denote the set of characters $\chi\in\irr(G)$ such that $N$ is not contained in the kernel of $\chi$.

\begin{prop}
\label{prop-classical} 
Let $H$ be a simple classical group. Suppose that
Hypothesis~\ref{hypothesis} holds for $H$. Then
Conjecture~\ref{conj:C} holds for $H$.
\end{prop}

\begin{proof}
Recall the setup in Conjecture~\ref{conj:C} that $G$ is a perfect
group having a normal elementary abelian $p$-group $N$ for some
prime $p$ dividing $|H|$ such that $G/N\cong H$. We wish to show
that there exists a faithful character $\chi\in\Irr(G)$ such that
$\chi(1)_p<|N|$. For this purpose, by Proposition~\ref{prop-tech},
we may assume that the defining characteristic of $H$ is $p$.

Let $P/N$ be a Sylow $p$-subgroup of $G/N$. Let $M$ be a maximal
subgroup of $N$ (so $M$ is also elementary abelian whose rank is one
less than that of $N$) such that $M$ is normal in $P$. By
hypothesis, there exists $\psi\in\Irr(P/M|N/M)$ such that $\psi(1)$
is less than the size of the natural $\F_q$-module $V$ for $H$.
Notice also that $|V|\leq|N|$, so $\psi(1)<|N|$.

Let $\lambda\in\Irr(N)$ be a (non-principal) character lying under
$\psi$ (here $\psi$ is viewed as a character of $P$). Since
$\psi^G(1)_p=\psi(1)<|N|$, there exists $\chi\in\Irr(G)$ lying over
$\psi$ such that $\chi(1)_p<|N|$. Since $\chi$ lies over $\psi$,
$\chi$ also lies over $\lambda$. We conclude that
$\chi\in\Irr(G|N)$, so $\chi$ is faithful, as wanted.
\end{proof}

E. O'Brien has checked with computer that
Hypothesis~\ref{hypothesis} holds when $H=\Sp_8(2)$. In fact, every
central extension of a Sylow $2$-subgroup $P$ of $\Sp_8(2)$ has a
character $\psi\in\Irr(P|Z)$ with $\psi(1)\leq2^6$. His computations
also suggest that Hypothesis~\ref{hypothesis} holds when
$H=\Sp_{2n}(2)$ for $5\leq n\leq 8$, but in these cases there are
too many central extensions of the Sylow $2$-subgroups. According to
these computations, perhaps it is true that there always exists
$\psi\in\Irr(P|Z)$ such that its degree is at most $2^{d-2}$, where
$d$ is the dimension of the natural module for $H$.


\end{document}